\documentclass[longbibliography]{amsart}
\usepackage[misc]{ifsym} 
\usepackage{lipsum}
\usepackage{amssymb}
\usepackage{amsmath}
\usepackage{latexsym}
\usepackage{bm}
\usepackage{enumitem}
\usepackage{graphicx}
\usepackage{amscd}
\usepackage{extarrows}
\usepackage{amsfonts}
\usepackage{amssymb}
\usepackage{indentfirst}
\usepackage{bbm}
\usepackage{mathdots}
\usepackage{mathtools}
\usepackage[linktocpage]{hyperref}
\usepackage{amsthm}
\hypersetup{
    colorlinks,
    citecolor=black,
    filecolor=black,
    linkcolor=blue,
    urlcolor=black
}

\usepackage{quiver}
\usepackage[multiple]{footmisc}
\usepackage[title]{appendix}

\newcommand\blfootnote[1]{%
  \begingroup
  \renewcommand\thefootnote{}\footnote{#1}%
  \addtocounter{footnote}{-1}%
  \endgroup
}

\usepackage{float}

\theoremstyle{plain}
\newtheorem{theorem}{Theorem}[section]
\newtheorem{proposition}[theorem]{Proposition}
\newtheorem{lemma}[theorem]{Lemma}
\newtheorem{corollary}[theorem]{Corollary}
\newtheorem*{theorem*}{Theorem}

\theoremstyle{remark}
\newtheorem{remark}[theorem]{Remark}

\theoremstyle{definition}
\newtheorem{definition}{Definition}[section]

\DeclareMathOperator{\Hom}{Hom}

\DeclareMathOperator{\tr}{tr}

\DeclareMathOperator{\covol}{covol}
\DeclareMathOperator{\vol}{vol}
\DeclareMathOperator{\Tr}{Tr}

\DeclareMathOperator{\id}{id}

\DeclareMathOperator{\SL}{SL}
\DeclareMathOperator{\GL}{GL}
\DeclareMathOperator{\PSL}{PSL}

\title{Von Neumann Dimensions and Trace Formulas I: Limit Multiplicities}

\author{Jun Yang}

\date{}

\begin{document}
\maketitle


\blfootnote{\Letter~Jun Yang~~\href{mailto:junyang@fas.harvard.edu}{junyang@fas.harvard.edu}~~~Harvard University, Cambridge, MA, USA}
\blfootnote{This work was supported in part by the ARO Grant W911NF-19-1-0302 and the ARO MURI Grant W911NF-20-1-0082.}

\begin{abstract}
Given a connected semisimple Lie group $G$ and an arithmetic subgroup $\Gamma$, it is well-known that each irreducible representation $\pi$ of $G$ occurs in the discrete spectrum $L^2_{\text{disc}}(\Gamma\backslash G)$ of $L^2(\Gamma\backslash G)$ with at most a finite multiplicity $m_{\Gamma}(\pi)$. 
While $m_{\Gamma}(\pi)$ is unknown in general, we are interested in its limit as $\Gamma$ is taken to be in a tower of lattices $\Gamma_1\supset \Gamma_2\supset\dots$.
For a bounded measurable subset $X$ of the unitary dual $\widehat{G}$, we let $m_{\Gamma_n}(X)$ be the sum of the multiplicity $m_{\Gamma_n}(\pi)$ over all $\pi$ in $X$. 
Let $H_X$ be the direct integral of the irreducible representations in $X$ with respect to the Plancherel measure of $G$, which is also a module over the group von Neumann algebra $\mathcal{L}\Gamma_n$. 
We prove:
\begin{center}
$\lim\limits_{n\to \infty}\cfrac{m_{\Gamma_n}(X)}{\dim_{\mathcal{L}\Gamma_n}H_X}=1$,
\end{center}
for any bounded subset $X$ of $\widehat{G}$, 
when i) $\Gamma_n$'s are cocompact, or, ii) $G=\SL(n,\mathbb{R})$ and $\{\Gamma_n\}$ are principal congruence subgroups.   
\end{abstract}

\tableofcontents

\section{Introduction: an example on $\SL(2,\mathbb{R})$}

In this section, we introduce a multiplicity problem of square-integrable irreducible representations of $G=\SL(2,\mathbb{R})$ on $L^2_{\text{cusp}}(\Gamma\backslash G)$ for some arithmetic subgroups $\Gamma$.
It is one of the motivations of this article. 

We first $\Gamma=\SL(2,\mathbb{Z})$ and $\Gamma(N)$ be the principal congruence subgroup of level $n$ defined by
\begin{center}
    $\Gamma(N)=\Bigl\{ \begin{pmatrix}
a & b \\
c & d
\end{pmatrix}\in \SL(2,\mathbb{Z}): a,d\equiv 1 \pmod{N},b,c\equiv 0 \pmod{N}\Bigl\}$. 
\end{center}
Consider the right quasi-regular representation of $G$ on $L^2(\Gamma(N)\backslash G)$ given by $(R(g)\phi)(x)=\phi(xg)$ for $\phi\in L^2(\Gamma(N)\backslash G)$, $g\in G$. 
It is well known (see \cite{Kn97tf}) to be reducible and can be decomposed as
\begin{center}
   $L^2(\Gamma(N)\backslash G)=L^2_{\text{cusp}}(\Gamma(N)\backslash G)\oplus L^2_{\text{cont}}(\Gamma(N)\backslash G)\oplus \mathbb{C}$.   
\end{center}
Here $L^2_{\text{cusp}}(\Gamma(N)\backslash G)$ is the cuspidal part, which is a direct sum of irreducible representations with finite multiplicities, i.e.,
\begin{center}
  $L^2_{\text{cusp}}(\Gamma(N)\backslash G)=\sum m_{\Gamma(N)}(\pi)\cdot \pi$,~~~~$m_{\Gamma(N)}(\pi)<\infty$ for each $\pi$, 
\end{center}
and $L^2_{\text{cont}}(\Gamma(N)\backslash G)$ is a direct integral of irreducible representations given by the Eisenstein series. 

The multiplicities $m_{\Gamma(N)}(\pi)$ are still unknown in general, except for some special families of irreducible representations including the discrete series of $\SL(2,\mathbb{R})$ (see \cite{Kn} for an introduction of discrete series). 
Let $S_k(\Gamma)$ be the space of cusp forms of weight $k$ for a Fuchsian group $\Gamma$. 
We have the following result (see \cite{Gel75} Theorem 2.10). 
\begin{lemma}\label{lmsl2}
For the discrete series $\pi_k$, we have $m_{\Gamma(N)}(\pi_k)=\dim S_k(\Gamma(N))$. 
\end{lemma}

By applying the dimension formulas of cusp forms (see \cite{DiaShur05} Chapter 3.9), we obtain
\begin{equation}\label{e1}
    m_{\Gamma(N)}(\pi_k)=(\frac{k-1}{24}-\frac{1}{4N})N^3\prod_{p|N}(1-\frac{1}{p^2})
\end{equation}
for all $N>2$. 

On the other hand, let $H_k$ be the underlying Hilbert space of the discrete series $\pi_k$. 
As $H_k$ is a module over the group $\Gamma(N)$, we can further prove that it is also a module over the {\it group von Neumann algebra} $\mathcal{L}(\Gamma(N))$ (see Section \ref{ssvnalg} for the definition). 
Hence $H_k$ has a {\it von Neumann dimension} $\dim_{\mathcal{L}(\Gamma(N))}H_k$ over $\mathcal{L}(\Gamma(N))$. 
Indeed, if a discrete group $\Gamma$ is ICC (infinite conjugacy class, see also \ref{ssvnalg}), this dimension totally determines the equivalence class of $\mathcal{L}(\Gamma)$-module, i.e., $\dim_{\mathcal{L}(\Gamma)}H_1=\dim_{\mathcal{L}(\Gamma)}H_2$ if and only if $H_1,H_2$ are isomorphic as $\mathcal{L}(\Gamma)$-module. 

We consider a lattice $\Gamma$ in a Lie group $G$. 
Suppose $(\pi,H)$ is a discrete series representation of $G$ and 
let $d(\pi)$ be the formal dimension of $\pi$ (see \cite{Ro} Chapter 16). 
We have
\begin{lemma}[Goodman-la Harpe-Jones\cite{GHJ}]
$\dim_{\mathcal{L}(\Gamma)}H=\vol(\Gamma\backslash G)\cdot d(\pi)$  
\end{lemma}

By  Example 3.3.4 in \cite{GHJ}, we know $\dim_{\mathcal{L}(\PSL(2,\mathbb{Z}))}H_k=\frac{k-1}{12}$.
As $\SL(2,\mathbb{Z})=(\mathbb{Z}/2\mathbb{Z})\rtimes \PSL(2,\mathbb{Z})$, 
we have $\dim_{\mathcal{L}(\SL(2,\mathbb{Z}))}H_k=\frac{k-1}{24}$. 
Since $[\Gamma\colon \Gamma(n)]=N^3\prod_{p|N}(1-\frac{1}{p^2})$, we can conclude 
\begin{equation}\label{e2}
    \dim_{\mathcal{L}(\Gamma(N))}H_k=\frac{k-1}{24}N^3\prod_{p|N}(1-\frac{1}{p^2}). 
\end{equation}
Thus we obtain:
\begin{corollary}\label{cexmp}
For a discrete series $(\pi_k,H_k)$ of $\SL(2,\mathbb{R})$, we have
\begin{center}
  $\lim_{N\to \infty}\cfrac{m_{\Gamma(N)}(\pi_k)}{\dim_{\mathcal{L}(\Gamma(N))}H_k}=1$.  
\end{center}
\end{corollary}
\begin{proof}
Comparing Equations \ref{e1} and \ref{e2}, we obtain $\frac{m_{\Gamma(N)}(\pi_k)}{\dim_{\mathcal{L}(\Gamma(N))}H_k}=\frac{k-1-6/N}{k-1}$ and then take the limit. 
\end{proof}

While the explicit multiplicities of most irreducible representations are still unknown, 
the limit multiplicities have been studied since 1970s. 
In the case of towers of uniform lattices, 
DeGeorge and Wallach got the first results for discrete series of Lie groups  \cite{dGW78} and later for bounded sets of irreducible representations in rank one group \cite{dGW79}. 
Delorme \cite{Delm86} finally solved the problem for bounded sets of irreducible representations in all Lie groups. 
See also \cite{ABBGNRS17} for a recent approach. 

For the non-uniform lattices (or most arithmetic subgroups), Savin \cite{Sav89} obtained the results on discrete series in his thesis at first, which is based on the work by Rohlfs and Speh \cite{RoSp87}. 
Then Deitmar and Hoffmann proved the results on certain towers of arithmetic subgroups in rank one group. 
Recently, Finis and Lapid solved the case of congruence subgroups in $\SL(n,\mathbb{R})$ \cite{FLM15,FLM18}, which are based on their study of the spectral side of Arthur's trace formulas \cite{FLM11a,FLM11b}. 

The goal of this paper is to extend Corollary \ref{cexmp} to some general settings. 
In the rest of this paper, we generalize this result mainly in the following aspects:
\begin{enumerate}
    \item from a single discrete series representation to any bounded subset of the unitary dual $\widehat{G}$ of $G$;
    \item from $\SL(2,\mathbb{R})$ to the towers of uniform lattices in an arbitrary semisimple Lie group,
    \item from $\SL(2,\mathbb{R})$ to  $\SL(n,\mathbb{R})$ with its the principal congruence subgroups. 
\end{enumerate} 
Finally, we are able to prove:
\begin{theorem}[\textbf{The Main Theorem}]
Let $G$ be a semisimple simply-connected Lie group. 
Let $X$ be a bounded subset of the unitary dual of $G$ and 
$H_X$ be the direct integral of the irreducible representations of $G$ in $X$. 
We have:
\begin{center}
$\lim\limits_{n\to \infty}\cfrac{m_{\Gamma_n}(X)}{\dim_{\mathcal{L}\Gamma_n}H_X}=1$,
\end{center}
when i) $\Gamma_n$'s are cocompact, or ii) $G=\SL(n,\mathbb{R})$ and $\{\Gamma_n\}$ are principal congruence subgroups.  
\end{theorem}

\section{The trace formulas and dominant terms}\label{stf}
 
We have a brief review of the Arthur-Selberg trace formulas and give the dominant terms in these formulas. 
We mainly follow \cite{ArtI,ArtII,FLM11}.

Let $\mathbf{G}$ be a reductive group over $\mathbb{Q}$. 
The group $G(\mathbb{A})$ acts naturally on $L^2(G(\mathbb{Q})\backslash G(\mathbb{A}))$ by 
\begin{center}
$R(g)\phi(x)=\phi(xg)$ 
\end{center}
for $\phi\in L^2(G(\mathbb{Q})\backslash G(\mathbb{A}))$ and $g\in G(\mathbb{A}))$.
Let $C^{\infty}_{\text{c}}{(G(\mathbb{A}))}$ be the complex algebra of smooth, compactly supported function on $G(\mathbb{A}))$.
Given $f\in C^{\infty}_{\text{c}}{(G(\mathbb{A}))}$, we may define
\begin{center}
    $(R(f)\phi)(x)=\int_{G(\mathbb{A}))}f(g)R(g)\phi(x)dg=\int_{G(\mathbb{A}))}f(g)\phi(xg)dg$. 
\end{center}
If we define the {\it kernel}
\begin{center}
$K(x,y)=K_f(x,y)\colon=\sum\limits_{\gamma\in G(\mathbb{Q})}f(x^{-1}\gamma y)$,
\end{center}
we have $(R(f)\phi)(x)=\int_{G(\mathbb{Q})\backslash G(\mathbb{A})}K(x,y)\phi(y)dy$. 

\subsection{The Selberg trace formula}\label{sstfS}

We first assume $\mathbf{G}$ is anisotropic and hence the quotient space $G(\mathbb{Q})\backslash G(\mathbb{A})$ is compact. 
Let $\mathcal{O}$ be the set of conjugacy classes in $G(\mathbb{Q})$ and $o\in \mathcal{O}$ be a conjugacy class. 
We may define
\begin{center}
    $K_o(x,y)=\sum\limits_{\gamma\in o}f(x^{-1}\gamma y)$
\end{center}
and obtain $K(x,y)=\sum\limits_{o\in \mathcal{O}} K_o(x,y)$. 
On the other hand, the representation $R$ decomposes into a direct sum of irreducible representations with finite multiplicities, i.e.,
$L^2(G(\mathbb{Q})\backslash G(\mathbb{A}))=\oplus_{\chi\in\mathcal{X}}L^2(G(\mathbb{Q})\backslash G(\mathbb{A}))_{\chi}$. 
Here $L^2(G(\mathbb{Q})\backslash G(\mathbb{A}))_{\chi}=m(\chi)\cdot \chi$, which is $m(\chi)$ copies of the irreducible representation $\chi$. 
Assume $\mathcal{B}_{\chi}$ is a orthonormal basis of $L^2(G(\mathbb{Q})\backslash G(\mathbb{A}))_{\chi}$. 
Then 
\begin{center}
$K_{\chi}(x,y)=K_{f,\chi}(x,y)\colon=\sum_{\phi\in \mathcal{B}_{\chi}}(R(f))\phi(x)\cdot\overline{\phi(y)}$   
\end{center}
converges. 
Now we let
\begin{enumerate}
\item $k_{\chi}(x,f)=K_{\chi}(x,x)$ and $J_{\chi}(f)=\int_{G(\mathbb{Q})\backslash G(\mathbb{A})}k_{\chi}(x,f)dx$,
    \item $k_o(x,f)=K_o(x,x)$ and $J_{o}(f)=\int_{G(\mathbb{Q})\backslash G(\mathbb{A})}k_o(x,f)dx$.
\end{enumerate}
If we let ${\gamma}$ be a representatives of $o\in \mathcal{O}$ and $H_{\gamma}=\{h\in H|h\gamma h^{-1}=\gamma\}$ for a group $H$ containing $\gamma$, we get
\begin{center}
$J_o(f)=\vol(G(\mathbb{Q})_{\gamma}\backslash G(\mathbb{A})_{\gamma})\int_{G(\mathbb{A})_{\gamma}\backslash G(\mathbb{A})}f(x^{-1}\gamma x)dx$. 
\end{center}

\begin{theorem}\label{ttfS}
Assuming $G(\mathbb{Q})\backslash G(\mathbb{A})$ is compact, we have
\begin{equation}\label{etfcpt}
\begin{aligned}
    \tr R(f)=\sum\limits_{o\in O}J_o(f)=\sum\limits_{\chi\in \mathcal{X}}J_{\chi}(f)
\end{aligned}
\end{equation}
for any $f\in C^{\infty}_c{(G(\mathbb{A}))}$.
\end{theorem}

For the classical setting, 
we start with a real Lie group $G$ and a lattice $\Gamma\subset G$. 
Consider the representation $R_{\Gamma}$ of $G$ on $L^2(\Gamma\backslash G)$ given by $(R_{\Gamma}(g)\phi)(x)=\phi(xg)$ for $x,g\in G$. 
Let $C^{\infty}_{c}(G)$ be the space of smooth function on $G$ with compact support. 
For $f\in C^{\infty}_{c}(G)$ and a representation $(\pi,H)$ of $G$, we let
\begin{center}
    $\pi(f)v=\int_G f(g)\pi(g)v dg$.
\end{center}
If $\pi$ is irreducible, $\pi(f)$ is a trace class operator and we let $\theta_{\pi}(f)=\tr\pi(f)$. 
Note for the representation $R_{\Gamma}$, we have $(R_{\Gamma}(f)\phi)(x)=\int_G f(g)R_{\Gamma}(g)\phi(x)dg=\int_G f(g)\phi(xg)dg$. 

It is known that $\Gamma\backslash G$ is compact if and only if the reductive part of $\mathbf{G}$ is anisotropic (see \cite{PR94} Theorem 4.12). 
In this case, $L^2(\Gamma\backslash G)$ can be decomposed into a direct sum of irreducible representations of $G$ with each of finite multiplicity, i.e., 
\begin{center}
    $L^2(\Gamma\backslash G)=\oplus m_{\Gamma}(\pi)\cdot \pi$
\end{center}
with $m_{\Gamma}(\pi)=\dim \Hom_G(\pi,L^2(\Gamma\backslash G))<\infty$ for each $\pi$.

By taking the test function in Theorem \ref{ttfS} to be $f\otimes 1_K$ for a maximal compact subgroup $K$ of $G(\mathbb{A}^{\text{fin}})$ with $f\in C^{\infty}_c(G)$ (see Section \ref{ssmult}), we get the following result for the lattice $\Gamma$ in the real Lie group $G$. 
\begin{corollary}[The Selberg trace formula]\label{ctfS}
If $\Gamma\backslash G$ is compact, $R_{\Gamma}(f)$ is of trace class and 
\begin{equation}\label{etflie}
    tr R_{\Gamma}(f)=\sum_{\pi\in \widehat{G}}m_{\Gamma}(\pi)\theta_{\pi}(f)=\sum_{\gamma\in [\Gamma]}\vol(\Gamma_{\gamma}\backslash G_{\gamma})\int_{\Gamma_{\gamma}\backslash G}f(x^{-1}\gamma x)dx
\end{equation}

\end{corollary}

\subsection{The Arthur trace formula}\label{sstfA}

We assume $\mathbf{G}$ is not necessarily anisotropic and $G(\mathbb{Q})\backslash G(\mathbb{A})$ may not be compact. 
Assume $B$ is a Borel subgroup defined over $\mathbb{Q}$,  
$M_0$ is a Levi factor of $B$ defined over $\mathbb{Q}$, 
$P$ is a standard parabolic subgroup defined over $\mathbb{Q}$ (i.e., $P_0=B\subset P$), 
$N_P=R_u(P)$ (the unipotent radical of $P$), 
$M_P$ is the unique Levi component of $P$ such that $M_0\subset M_P$. 

We also assume $A_P$ is the split component of the center of $M_P$ and $Z=A_G$, $\Delta_0=\Delta_B$ is a base for a root system. 
We will mostly use the notations of \cite{ArtI,ArtII,ArtIntro} and \cite{GelTF} as follows:

\begin{itemize}
\small
\item $\mathfrak{a}_P=\Hom(X(M_P)_{\mathbb{Q}},\mathbb{R})$ where $X(M_P)_{\mathbb{Q}}$ is the $\mathbb{Q}$-characters of $M_P$, $\mathfrak{a}_P^*=X(M_P)_{\mathbb{Q}}\otimes \mathbb{R}$ and $\mathfrak{a}_P^+=\{H\in \mathfrak{a}_P|\alpha(H)>0,\forall \alpha\in \Delta_P\}$. 
\item $\gamma=\gamma_s \gamma_u$, which is the decomposition such that $\gamma_s$ is semisimple and $\gamma_u$ is unipotent. 
 \item $\mathcal{O}$ is the set of $G(\mathbb{Q})$-semisimple conjugacy class of $G(\mathbb{Q})$ ($\gamma\cong \beta$ if $\gamma_s$ and $\beta_s$ are $G(\mathbb{Q})$-conjugate). 
 \item $o\in \mathcal{O}$ is a conjugacy class in $G(\mathbb{Q})$.
 \item $\mathcal{X}$ is the set of equivalence classes of pairs $(M,\rho)$, where $M$ is a Levi subgroup of $G$ and $\rho\in \widehat{M(\mathbb{A})^1}$ ($(M,\rho)\sim (M',\rho')$ if there is an $s\in \Omega(\mathfrak{a},\mathfrak{a}')$ such that the representation $(s\rho)(m')=\rho(w_s^{-1}m w_s)$ is unitarily equivalent to $\rho'$). 
\item For a pair of parabolic groups $P_1\subset P_2$, $\Delta^{P_2}_{P_1}$ is the set of simple roots of $(M_{P_2}\cap P_1,A_{P_1})$ and $\hat{\Delta}^{P_2}_{P_1}=\{\varpi_{\alpha}|\alpha\in \Delta^{P_2}_{P_1}\}$, i.e, the dual basis for $\Delta^{P_2}_{P_1}$. 
\item $\hat{\tau}_P$ is the characteristic function on $\mathfrak{a}_0$ of $\{H\in \mathfrak{a}_0| \varpi(H)>0,\varpi\in \hat{\Delta}^G_P\}$. 
\item For $m=\prod_v m_v\in M(\mathbb{A})$, let $H_M(m)\in \mathfrak{a}_p$ given by 
\begin{center}
    $e^{\langle H_M(m),\chi\rangle}=|\chi(m)|=\prod\limits_{v}|\chi(m_v)|_v$, $\forall \chi\in X(M)_{\mathbb{Q}}$. 
\end{center}
    \item $x=nmak\in G(\mathbb{A})$ with $n\in G(\mathbb{A}), m\in M(\mathbb{A})^1,a\in A(\mathbb{R})^0$ and $k\in K$.
    \item $H(x)=H_M(ma)=H_M(a)\in \mathfrak{a}_p$. 
\end{itemize}

Let $T\in \mathfrak{a}_0^+$ be suitably regular, i.e., $\alpha(T)$ is sufficiently large for all $\alpha\in \Delta_0$. 
For a parabolic subgroup $P$, 
there are kernels $K_{P,o}=\sum\limits_{\gamma\in M(\mathbb{Q})\cap o}\int_{N(\mathbb{A})}f(x^{-1}\gamma n y)dn$ 
and $K_{P,\chi}$ (see \cite{ArtI} p.923 and p.935 for the precise definitions). 
Then Arthur is able to define the {\it truncated kernels} and distributions $J^T_{o},J^T_{\chi}$ as follows:
\begin{enumerate}
\item $k^T_{o}(x,f)=\sum\limits_{P}(-1)\dim(A_P/Z)\sum\limits_{\delta\in P(\mathbb{Q})\backslash G(\mathbb{Q})}K_{P,o}(\delta x,\delta x)\cdot \hat{\tau}_p(H(\delta x)-T)$. 
    \item $k^T_{\chi}(x,f)=\sum\limits_{P}(-1)\dim(A_P/Z)\sum\limits_{\delta\in P(\mathbb{Q})\backslash G(\mathbb{Q})}K_{P,\chi}(\delta x,\delta x)\cdot \hat{\tau}_p(H(\delta x)-T)$. 
    \item $J^T_{o}(f)=\int_{G(\mathbb{Q})\backslash G(\mathbb{A})^1}k^T_{o}(x,f)dx$. 
    \item $J^T_{\chi}(f)=\int_{G(\mathbb{Q})\backslash G(\mathbb{A})^1}k^T_{\chi}(x,f)dx$. 
\end{enumerate}
Let $\mathcal{X}(G)=\{(M,\rho) \in \mathcal{X}|M=G\}$.
We reach a coarse trace formula, which is firstly given in \cite{ArtII} Chapter 5. 
\begin{theorem}[The Arthur trace formula]\label{tAtf}
For any $f\in C^{\infty}_{c}(G(\mathbb{A})^1)$ and any suitably regular $T\in \mathfrak{a}_0^+$, we have
\begin{equation}\label{etf}
\begin{aligned}
\sum\limits_{o\in\mathcal{O}}J^T_{o}(f)=\sum\limits_{\chi\in\mathcal{X}}J^T_{\chi}(f) 
\end{aligned}
\end{equation}
Moreover, the trace formula of $R(f)$ is given by
\begin{center}
    $\tr R_{\text{cusp}}(f)=\sum\limits_{o\in\mathcal{O}}J^T_{o}(f)-\sum\limits_{\chi\in\mathcal{X}\backslash \mathcal{X}(G)}J^T_{\chi}(f)$.
\end{center}
\end{theorem}

\subsection{The dominant term on the geometric side}

Consider the adelic case at first. 
Let $F$ be a number field and $V,V_{\infty}$ and $V_f$ be the set of places, Archimedean and non-Archimedean places of $F$ respectively. 
Let $\mathbb{A}$ be adele ring of $F$ and $A_{\text{fin}}\subset\mathbb{A} $ be restricted product over the finite places. 

Suppose $S\subset V$ is a finite set containing $V_{\infty}$. 
Let $F_S=\prod_{v\in S}F_s$ and $\mathbb{A}^S=\prod'_{v\in V\backslash S}F_s$ so that $\mathbb{A}=F_S\times \mathbb{A}^S$. 
We define
\begin{enumerate}
    \item $G(F_S)^1=\underset{\chi\in \Hom(G(F_S),F^{\times})}{\bigcap}\{\ker|\chi|\colon G(F_S)\to \mathbb{R}_{+}\}$,
    \item $G(\mathbb{A})^1=\underset{\chi\in \Hom(G(\mathbb{A}),F^{\times})}{\bigcap}\{\ker|\chi|\colon G(\mathbb{A})\to \mathbb{R}_{+}\}$,
\end{enumerate}
where $|\cdot|$ is the product of valuations on $F_S$ and $\mathbb{A}$ respectively. 
 
We will consider the representation of $G(F_S)$ on $L^2(G(F)\backslash G(\mathbb{A})^1/K)$ for an open compact subgroup $K$ of $G(\mathbb{A}^S)$. 
In particular, it will reduce to the representation of $G(F_{\infty})$ on $L^2(\Gamma_K\backslash F_{\infty})$ if we take $S=\{\infty\}$ and $\Gamma_K=G(F)\cap K$.

Let $J(f)$ be the distribution defined by Equation \ref{etfcpt} or \ref{etf} in Section \ref{stf} for $f\in C^{\infty}_{c}(G(F_S))$, which also depends on $G(F)\backslash G(\mathbb{A})^1$ is compact or not. 
The goal of this subsection is to prove
\begin{center}
    $\lim\limits_{n\to \infty}\cfrac{\vol(G(F)\backslash G(\mathbb{A})^1)f(1)}{J(f\otimes 1_{K_n})}=1$ 
\end{center}
for certain towers of open compact subgroups $\{K_n\}_{n\geq 1}$ of $G(\mathbb{A}^S)$. 



Let us assume $\Gamma$ is a uniform lattice in the semisimple Lie group $G$.  
We add a subscript such as $R_{\Gamma}$ and $J_{\Gamma}$ for the representation of $G$ on $L^2(\Gamma\backslash G)$ and the corresponding trace formulas as an emphasis on the lattice $\Gamma$.  
Since $\Gamma\backslash G$ is compact, $J_{\Gamma}(f)$ is the trace $\tr R_{\Gamma}(f)$ and we obtain
\begin{center}
    $J_{\Gamma}(f)=\tr R_{\Gamma}(f)=\sum_{\pi\in \widehat{G}}J_{\pi,\Gamma}(f)=\sum_{o\in \mathcal{O}}J_{o,\Gamma}(f)$. 
\end{center}
Let $J_{\{1\},\Gamma}(f)=\vol(\Gamma\backslash G)f(1)$, the contribution of the identity to the geometric side of the trace formula. 

We take a tower of uniform lattices $\{\Gamma_n\}_{n\geq 1}$, such that $\Gamma_n\trianglelefteq \Gamma_1$, $[\Gamma_1:\Gamma_n]<\infty$ and $\cap_{n\geq 1}\Gamma_n=\{1\}$. 


\begin{proposition}\label{pcpttr=}
With the assumption of uniform lattice $\{\Gamma_n\}$ above, we have
\begin{center}
    $\lim\limits_{n\to \infty}\cfrac{J_{\{1\},\Gamma_n}(f)}{J_{\Gamma_n}(f)}=1$. 
\end{center}  
\end{proposition}
\begin{proof} 
Following the Equation (2) in \cite{Corw77}, we obtain
\begin{center}
\small
$\tr R_{\Gamma_n}(\phi)=J_{\{1\},\Gamma_n}(\phi)+\sum_{\gamma\neq 1}s_n(\gamma)\vol(\Gamma_j\backslash G)\vol(\Gamma_\gamma\backslash G_{\gamma}) \int_{\Gamma_{\gamma}\backslash G}\phi(x^{-1}\gamma x)dx$, 
\end{center}
where $0\leq s_n(\gamma) \leq \vol(\Gamma\backslash G)^{-1}$. 
As $\cap_{n\geq 1}\Gamma_n=\{1\}$, $\lim_{n\to \infty}s_n(\gamma)$ for all $\gamma\neq 1$. 
By \cite{Corw77} Theorem 2, we have $\vol(\Gamma_n\backslash G)^{-1}\cdot \lim_{n\to \infty}\tr R_{\Gamma_n}(\phi)=\phi(1)$. 
Hence $\lim\limits_{n\to \infty}\frac{J_{\{1\},\Gamma_n}(\phi)}{J_{\Gamma_n}(\phi)}=\lim\limits_{n\to \infty}\frac{J_{\{1\},\Gamma_n}(\phi)}{\tr R_{\Gamma_n}(\phi)}=1$. 
\end{proof}

Now we let $G$ be a reductive group over a number field $F$. 
Let $K=K_{\infty}K_{\text{fin}}$ be a maximal compact subgroup of $G(\mathbb{A})=G(\mathbb{A}_F)$. 
By fixing a faithful $F$-rational representation $\rho\colon G(F)\to \GL(m,F)$ for some $m>0$, we let $\Lambda\subset F^m$ be an $\mathcal{O}_F$-lattice such that the stablilizer of $\widehat{\Lambda}=\mathcal{O}_F\otimes_F \Lambda$ in $G(A_{\text{fin}})$ is $K_{\text{fin}}$. 

For a non-trivial ideal $I$ of $\mathcal{O}_F$, we let
\begin{center}
    $K(I)=\{g\in G(A_{\text{fin}})|\rho(g)v\equiv v \pmod{I\cdot\widehat{\Lambda}},v\in \widehat{\Lambda}\}$
\end{center}
be the {\it principal congruence subgroup} of level $I$.  
We also denote the ideal norm of $I$ by $N(I)=[\mathcal{O}_F\colon I]$. 
Consider a descending tower of ideals $I_1\supsetneq I_2\supsetneq I_3\supsetneq\cdots$ such that each $I_k$ is prime to (the prime ideals in) $S$. 
We obtain the corresponding tower of principal congruence subgroups: 
\begin{center}
     $K_{1}\supsetneq K_{2} \supsetneq K_{3}\supsetneq\cdots$,
\end{center}
where $K_{n}=K(I_n)$. 
By factoring into prime ideals, the family $\{I_n\}_{n\geq 1}$ satisfies either one of the following properties:
\begin{enumerate}
    \item there exists a prime ideal $\mathfrak{p}$ such that each $\mathfrak{p}^k$ is eventually contained in the tower, i.e., for any $k\geq 1$, there is $N_k>0$ such that $\mathfrak{p}^k\subset I_n$ for all $n\geq n_k$, or,
    \item there exists infinitely many prime ideals $\{\mathfrak{p}_k\}_{k\geq 1}$ such that for each $k$, there exist $M_k>0$ such that  $\mathfrak{p}_k\subset I_n$ for all $n\geq M_k$.  
\end{enumerate}
In either of these two cases, we have \begin{lemma}\label{lcap=1}
     $\cap_{n\geq 1}I_n=\{0\}$ and $\cap_{n\geq 1}K_{n}=\{1\}$.
\end{lemma}

Recall the equivalence class of unipotent elements in $G(F)$,  which is the element $\gamma=\gamma_s\gamma_u$ with the semisimple component $r_s=1$ (see \cite{ArtUni} p.1240). 
Let 
\begin{center}
 $J_{\text{unip}}^T(f)$, $f\in C^{\infty}_{c}(G(\mathbb{A})^1)$. 
\end{center}
be the contribution of this equivalence class on the geometric side of the trace formula \ref{etf}. 
We will consider the function of the form $f=h_S\otimes 1_{K_{n}}$ with $h_S\in C^{\infty}_{c}(G(F_S)^1)$. 

\begin{lemma}\label{lunipcls}
For $h_S\in C^{\infty}_{c}(G(F_S)^1)$, 
$\lim\limits_{n\to \infty}J(h_S\otimes 1_{K_{n}})=\lim\limits_{n\to \infty}J_{\text{unip}}(h_S\otimes 1_{K_{n}})$. 
\end{lemma}
\begin{proof}
Let $D_h=\text{supp}(h_S)\subset G(F_S)^1$ be the compact support of $h_S$. 
Then $\text{supp}(h_S\otimes 1_{K_{n})})=D_h K_{n}$ is compact and hence it intersects finitely many semisimple-conjugate class $o\in \mathcal{O}$. 

Consider the trace formula and Equation \ref{etf}, only the classes $o$'s (and its $G(\mathbb{A})$-conjugations) which intersect infinitely many $D_h K_{n}$ contributes a non-trivial $J_o(h_S\otimes 1_{K_{n}})$ to the limit $\lim\limits_{n\to \infty}J(h_S\otimes 1_{K_{n}})$. 

Suppose the $G(\mathbb{A})$ conjugacy classes of elements in $o$ intersects $D_h K_{n}$ for infinitely many $n$, i.e.,
$\{g\gamma g^{-1}|g\in G(\mathbb{A}),\gamma\in o \}\cap D_h K_{n}\neq \emptyset$ for infinitely many $n$. 
Take some $\gamma\in o$. 
By fixing a faithful $F$-representation $\rho\colon G(F)\to \GL(m)$, we let $p(x)\in F[x]$ be the characteristic polynomial of $\rho(\gamma)-1$ (a $m$-by-$m$ matrix over $F$). 
Suppose $p(x)=x^m+a_{m-1}x^{m-1}+\cdots+a_0$ with all $a_i\in F$. 
By Lemma \ref{lcap=1}, we know $a_i$ belongs to infinitely many $I_n$, or, equivalently $a_i=0$. 
Hence $p(x)=x^m$ and $\gamma$ is unipotent. 
\end{proof}

The unipotent contribution $J_{\text{unip}}(h_S\otimes 1_{K_{n}})$ can be further reduced to the the one from the identity as follows. 
We let $I_S$ be a product of prime ideals in at the places of $S$ and $K_{S-S_{\infty}}(I_S)$ be the $S-S_{\infty}$ component of the compact group $K(I_S)$. 

We also let $C^{\infty}_{\Omega}(G(F_S)^1)$ be the set of smooth functions with compact support contained in a compact subset $\Omega$ of $G(F_S)^1$. 
For each $k\geq 0$, we let $\mathcal{B}_k$ be the $k$-th component of the universal enveloping algebra $\mathcal{U}(\mathfrak{g}_{\mathbb{C}})$, where $\mathfrak{g}_{\mathbb{C}}$ is the complexified Lie algebra of the Lie group $G(F_{\infty})$. 
We set $\|h\|_k=\sum_{X\in \mathcal{B}_k}\|X\circ h\|_{L^1(G(\mathbb{A})^1}$ for $h\in C^{\infty}_{\Omega}(G(F_S)^1)$. 

The following result is a special case of  Proposition 3.1 in \cite{FLM15}, whose proof is mainly based on Theorem 3.1 and Theorem 4.2 in \cite{ArtUni}.
\begin{proposition}[Finis-Lapid-M\"{uller}]
There exists an integer $k\geq 0$ such that for any compact subset $\Omega$ of $G(F_S)^1$, we have a constant $C_{\Omega}>0$ and
\begin{center}
    $|J_{\text{unip}}(h_S\otimes 1_{K_{n}})-\vol(G(F)\backslash G(\mathbb{A})^1)h_{S}(1)|\leq C_{\Omega}\frac{(1+\log N(I_SI))^{d_0}}{N(I)}\|h_S\|_k$
\end{center}
for any bi-$K_{S-S_{\infty}}(I_S)$-invariant function $h_S\in C^{\infty}_{\Omega}(G(F_S)^1)$. 
\end{proposition}
Then, combining Lemma \ref{lunipcls} we can obtain: 
\begin{corollary}\label{cgeolimit}
   For $h_S\in C^{\infty}_{c}(G(F_S)^1)$, we have
\begin{center}
    $\lim\limits_{n\to \infty}\cfrac{\vol(G(F)\backslash G(\mathbb{A})^1)h_S(1)}{J(h_S\otimes 1_{K^S(n)})}=1$. 
\end{center} 
\end{corollary}

\section{The multiplicities problem}\label{smult}

This section is devoted to the multiplicity of bounded subsets of the unitary dual, instead of a single irreducible representation. 
\subsection{The multiplicities in $L^2(\Gamma\backslash G)$}\label{ssmult}

Let $G=\mathbf{G}(\mathbb{R})^0$, the connected component of the real group obtained from an almost simple group $\mathbf{G}$ over $\mathbb{Q}$. 
By fixing a faithful $\mathbb{Q}$-embedding $\rho:\mathbf{G}\to GL_n$, we have an arithmetic group $\Gamma$ commensurable with $G\cap \GL_n(\mathbb{Z})$. 
Let $\widehat{G}$ be the unitary dual of $G$ and $\widehat{G}_{\text{temp}}\subset \widehat{G}$ be the tempered dual.
Let us consider the following two cases. 


1. {\it $\Gamma\backslash G$ is compact}. As introduced in Section \ref{sstfS}, $L^2(\Gamma\backslash G)$ can be decomposed into a direct sum of irreducible representations of $G$ with each of finite multiplicity, i.e., 
\begin{center}
    $L^2(\Gamma\backslash G)=\oplus m_{\Gamma}(\pi)\cdot \pi$
\end{center}
with $m_{\Gamma}(\pi)\colon =\dim \Hom_G(\pi,L^2(\Gamma\backslash G))<\infty$ for each $\pi \in \widehat{G}$. 

2. {\it $\Gamma\backslash G$ is not compact.} If $G$ is semisimple, we have $\Gamma\backslash G$ is of finite (Haar) measure (see \cite{PetVlt14} Theorem 4.13). 
The regular representation has both discrete and continuous spectra: 
$L^2(\Gamma\backslash G)=L^2_{\text{disc}}(\Gamma\backslash G)\oplus L^2_{\text{disc}}(\Gamma\backslash G)$. 
The discrete spectrum can be written as the direct sum of cuspidal and residue subspaces: $L^2_{\text{disc}}(\Gamma\backslash G)=L^2_{\text{cusp}}(\Gamma\backslash G)\oplus L^2_{\text{res}}(\Gamma\backslash G)$. 
which can be decomposed further into a direct sum of irreducible representations with finite multiplicities, i.e., 
\begin{center}
    $L^2_{\text{disc}}(\Gamma\backslash G)=\oplus m_{\Gamma}(\pi)\cdot \pi$
\end{center}
with $m_{\Gamma}(\pi)\colon=\dim \Hom_G(\pi,L^2_{\text{disc}}(\Gamma\backslash G))=\dim \Hom_G(\pi,L^2(\Gamma\backslash G))$ is finite for each $\pi \in \widehat{G}$. 

We say $X\subset \widehat{G}$ is {\it bounded} if it is relatively compact under the Fell topology. 

\begin{definition}[The multiplicity for $X\subset \widehat{G}$]
For a bounded $X \subset \widehat{G}$, we define the {\it multiplicity of $X$} to be the sum of the multiplicities of the irreducible representations in $X$, i.e., 
\begin{center}
    $m_{\Gamma}(X)\colon=\sum_{\pi \in X}m_{\Gamma}(\pi)$. 
\end{center}
\end{definition}

Borel and Garland proved the finiteness of $m_{\Gamma}(X)$ by considering the spectrum of a certain Laplacian (see \cite{BoGa83} Theorem 3, Theorem 4.6 and also \cite{Ji98} Theorem 1.1.3).
\begin{theorem}[Borel-Garland]\label{tBGfin}
Let $G=\mathbf{G}(\mathbb{R})^0$ for a connected semisimple group $\mathbf{G}$ over $\mathbb{Q}$ and $X\subset \widehat{G}$ is bounded. 
We have $m_{\Gamma}(X)<\infty$. 
\end{theorem}

For a subset $X\subset \widehat{G(F_S)^1}$, we call it {\it bounded} if it is relatively compact under the Fell topology (see \cite{Svg97}).

\begin{definition}[The multiplicity for $\widehat{G(F_S)^1}$]
Suppose $K$ is a compact open subgroup of $ G(\mathbb{A}^S)$. 
Let $\sigma$ be an irreducible representation of $G(F_S)^1$ and $X\subset \widehat{G(F_S)^1}$ be a bounded subset. 

\begin{enumerate}
    \item The {\it multiplicity of $\sigma$ with respect to $K$} is defined as 
    \begin{center}
        $m_{K}(\sigma)\colon=\dim \Hom_{G(F_S)^1}(\sigma,L^2(G(\mathbb{Q})\backslash G(\mathbb{A})^1/K))$.
    \end{center}
    \item The {\it multiplicity of $X$ with respect to $K$} is defined as 
\begin{center}
    $m_{K}(X)\colon=\sum_{\sigma\in X}m_{K}(\sigma)$. 
\end{center}
\end{enumerate}
\end{definition}

For an irreducible representation $\pi$ of $G(\mathbb{A})^1$, we write $\pi=\pi_S\otimes \pi^S$, where $\pi_S$ and $\pi^S$ denote the components of the representations of $G(F_S)^1$ and $G(\mathbb{A}^S)$ respectively. 
As shown in Theorem \ref{tBGfin}, $m_K(X)$ is finite and hence well-defined. 
If we treat $L^2(G(\mathbb{Q})\backslash G(\mathbb{A})^1/K)$ as the subspace of $K$-right invariant functions in $L^2(G(\mathbb{Q})\backslash G(\mathbb{A})^1))$, we have
\begin{center}
   $m_{K}(\sigma)=\sum_{\pi\in \widehat{G(\mathbb{A})^1},\pi_S=\sigma}\dim \Hom_{G(\mathbb{A})^1}(\pi,L^2(G(\mathbb{Q})\backslash G(\mathbb{A})^1))\dim(\pi^S)^K$. 
\end{center}
If we take $S=V_{\infty}$ and $\mathbf{G}$ is semisimple, simply connected, and without any $F$-simple factors $H$ such that $H(F_{\infty})$ is compact 
and $K$ is an open compact subgroup of $G(\mathbb{A}_{\text{fin}})$, 
we know $\Gamma_K=G(\mathbb{F})\cap K$ is a lattice in the seimisimple Lie group $G(F_{\infty})$. 

\begin{lemma}\label{l2m}
With the assumption above, we have $m_{\Gamma_K}(\pi)=m_K(\pi)$ for any $\pi \in \widehat{G(F_{\infty})^1}$ and  $m_{\Gamma_K}(X)=m_K(X)$ for any bounded $X \subset \widehat{G(F_{\infty})^1}$
\end{lemma}
\begin{proof}
It follows the fact $G(\mathbb{Q})\backslash G(\mathbb{A})/K$ can be identified with $\Gamma_K\backslash G(F_{\infty})$, which leads to a $G(F_{\infty})$-isomorphism $L^2(\Gamma_K\backslash G(F_{\infty}))\cong L^2(G(\mathbb{Q})\backslash G(\mathbb{A})^1/K)$ (see \cite{Kn97llds} Chapter 6 and \cite{PR94} Chapter 7.4.). 
\end{proof}

For a finite set $S$ and a function $\phi$ on $\widehat{G(F_S)^1}$, we define
\begin{center}
    $m_K(\phi)\colon =\int_{\widehat{G(F_S)^1}}\phi(\pi)dm_K(\pi)$
\end{center}
as its integral with respect to the measure given by multiplicities above. 
If $1_X$ is the characteristic function of $X$, i.e., $1_X(\pi)=1$ if $\pi \in X$ and $0$ otherwise, $m_K(1_X)=m_K(X)$. 

For $f\in C^{\infty}_{\text{c}}(G(F_S)^1)$, we let $\widehat{f}(\pi)=\tr \pi(f)$, the distribution character of $\pi$. 
Let $R_{\text{disc}}$ denote the action of $G(\mathbb{A})$ on the discrete subspace $L^2(G(\mathbb{Q})\backslash G(\mathbb{A})^1)$.

\begin{proposition}\label{ptrmk}
For $f\in C^{\infty}_{\text{c}}(G(F_S)^1)$, we have
\begin{center}
$\tr R_{\text{disc}}(f\otimes \frac{1_K}{\vol(K)})=m_K(\hat{f})$. 
\end{center}
\end{proposition}
\begin{proof}
Observe for the component $\pi^S$ of representation of $G(\mathbb{A}^S)$, we have
\begin{equation*}
\begin{aligned}
    \tr \pi^S(1_K)&=\int_{G(\mathbb{A}^S)}1_K(x)\pi^S(x^{-1})d\mu^S(x)\\
    &=\int_{K}\pi^S(x^{-1})d\mu^S(x)=\vol(K)\dim(\pi^S)^K,
\end{aligned}
\end{equation*}
where we apply the fact that $\int_K \sigma(x)d\mu^S(x)=0$ for any non-trivial irreducible representation $\sigma$ of $K$. 
Hence we obtain
\begin{equation*}
\begin{aligned}
  \tr R_{\text{disc}}(f\otimes \frac{1_K}{\vol(K)})&=\frac{1}{\vol(K)}\sum_{\pi\in \widehat{G(\mathbb{A})^1}}m(\pi)\tr \pi(f\otimes 1_K)\\
  &=\frac{1}{\vol(K)}\sum_{\pi\in \widehat{G(\mathbb{A})^1}}m(\pi)\tr \pi_S(f)\tr \pi^S(1_K)\\
  &=\frac{1}{\vol(K)}\sum_{\pi\in \widehat{G(\mathbb{A})^1}}m(\pi)\tr \pi_S(f)\vol(K)\dim(\pi^S)^K\\
  &=\sum_{\sigma\in \widehat{G(F_S)^1}}m_K(\sigma)\tr \sigma(f)=m_K(\widehat{f}). 
\end{aligned}
\end{equation*}
\end{proof}

We also give the following result which connects the trace formulas for adelic groups and Lie groups.
\begin{corollary}\label{c2trf}
 Let $\Gamma_K=G(F)\cap K$ with an open compact subgroup $K$ of $G(\mathbb{A}_{\rm{fin}})$. 
We have
\begin{center}
    $\tr R_{\text{disc}}(f\otimes \frac{1_K}{\vol(K)})=\tr R_{\Gamma_K}(f)$. 
\end{center}
for all $f\in C^{\infty}_{c}(G(F_{\infty})^1)$.    
\end{corollary}
\begin{proof}
It follows the fact $m_K(\widehat{f})=m_{\Gamma_K}(\widehat{f})$ in Lemma \ref{l2m},  $m_{\Gamma_K}(\widehat{f})=\tr R_{\Gamma_K}(f)$ and Proposition \ref{ptrmk}.  
\end{proof}

\subsection{Sauvageot's density theorems}\label{ssSvg}

We have a brief review of the results in \cite{Svg97}. 
See also \cite{Shin12} for an alternative approach and corrections. 

For an open compact subgroup $K$ of $G(\mathbb{A}^S)$, we define a measure on $\widehat{G(F_S)^1}$ by
\begin{center}
$\nu_K(X)\colon=\frac{\vol(K)}{\vol(G(\mathbb{Q})\backslash G(\mathbb{A})^1)}m_K(X)$
\end{center}
for any bounded subset $X$ of $\widehat{G(F_S)^1}$ and $m_K$ is the multipilicity defined in Chapter \ref{ssmult}.

Let $K_1\supsetneq K_2 \supsetneq \cdots $ be a sequence of open compact subgroups of $G(\mathbb{A}^S)$. 
Given a bounded subset $X$ of $\widehat{G(F_S)^1}$ and $C\geq 0$, we write 
\begin{center}
$\lim\limits_{n\to \infty}\nu_K(X)=C$, 
\end{center}
if for any $\varepsilon>0$, there exists  $N=N(\varepsilon)>0$ such that $|\nu_{K_n}(X)-C|< \varepsilon$ for all $n\geq N$. 

Let $\mathcal{H}(G(F_S)^1)$ be the complex algebra of smooth, compactly-supported, bi-$K_S$-finite functions on $G(F_S)^1$. 

\begin{lemma}[\cite{Svg97} Corollaire 6.2]\label{ls1}
For $\varepsilon>0$ and any bounded $X\subset \widehat{G(F_S)^1}\backslash \widehat{G(F_S)^1}_{\text{temp}}$, there is $\Psi\in \mathcal{H}(G(F_S)^1)$ such that
\begin{center}
$\widehat{\Psi}|_{\widehat{G(F_S)^1}}\geq 0$, $\nu(\widehat{\Psi})<\varepsilon$ and $\widehat{\Psi}|_{X}\geq 1$.
\end{center}
\end{lemma}
Given a function $f$ defined on $\widehat{G(F_S)^1}_{\text{temp}}$, we also denote by $f$ the function on $\widehat{G(F_S)^1}$, which is extended by $0$ on the untempered part.  
\begin{lemma}[\cite{Svg97} Theoreme 7.3(b)]\label{ls2}
For $\varepsilon>0$ and any $\nu$-integrable function $f$ on $\widehat{G(F_S)^1}_{\text{temp}}$, there exists $\phi,\psi\in \mathcal{H}(G(F_S)^1)$ such that
\begin{center}
$|f(\pi)-\widehat{\phi}(\pi)|\leq \widehat{\psi}(\pi)$ and $\nu(\widehat{\psi})<\varepsilon$. 
\end{center}
\end{lemma}

Here we obtain one of the main results in \cite{Svg97} and we also provide a proof for completeness. 
\begin{theorem}[Sauvageot]\label{tsvg}
Suppose $\lim\limits_{n\to \infty}\nu_{K_n}(\widehat{\phi})=\phi(1)$ for all $\phi\in \mathcal{H}(G(F_S)^1)$. 
We have
\begin{center}
$\lim\limits_{n\to \infty}\nu_{K_n}(X)=\nu(X)$
\end{center}
for all bounded subset $X$ of $\widehat{G(F_S)^1}$. 
\end{theorem}
\begin{proof}
First, we show the contribution from the untempered part is negligible in the limit. 
For a bounded subset $X_0$ of $\widehat{G(F_S)^1}\backslash \widehat{G(F_S)^1}_{\text{temp}}$ and $\varepsilon>0$, we let $\Psi\in \mathcal{H}(G(F_S)^1)$ satisfies Lemma \ref{ls1} with respect to $X$. 
We have $\nu_{K_n}(X)\leq \nu_{K_n}(\widehat{\Psi})\leq |\nu_{K_n}(\widehat{\Psi})-\psi(1)|+\psi(1)<2\varepsilon$ for all $n\geq N_1$ with some $N_1\geq 0$. 

For the tempered part, 
we fix a bounded subset $X_1$ of $\widehat{G(F_S)^1}_{\text{temp}}$ with the same $\varepsilon$ above.
Let $\phi,\psi\in \mathcal{H}(G(F_S)^1)$ satisfy Lemma \ref{ls2} with respect to the function $f=1_{X_1}$ on $\widehat{G(F_S)^1}_{\text{temp}}$ and $\varepsilon$. 
By assumption, we have $|\nu_{K_n}(\widehat{\phi})-\phi(1)|<\varepsilon$ and $|\nu_{K_n}(\widehat{\psi})-\psi(1)|<\varepsilon$ for all $n\geq N_2$ with some $N_2\geq 0$. 
Hence, for $n\geq N_2$, we obtain
\begin{equation*}
\begin{aligned}
|\nu_{K_n}(X_1)-\nu(X_1)|& \leq |\nu_{K_n}(X_1)-\nu_{K_n}(\widehat{\phi})|+|\nu_{K_n}(\widehat{\phi})-\phi(1)|+|\phi(1)-\nu(X)|\\
& \leq |\nu_{K_n}(\widehat{\phi})-\phi(1)|+\nu_{K_n}(\widehat{\psi})+\psi(1)\\
& \leq |\nu_{K_n}(\widehat{\phi})-\phi(1)|+|+|\nu_{K_n}(\widehat{\psi})-\psi(1)|+2\psi(1)<4\varepsilon.
\end{aligned}
\end{equation*}
Hence, for the bounded set $X$ of $\widehat{G(F_S)^1}$, let $X=X_0\sqcup X_1$ be the decomposition into its untempered and tempered parts. 
We have 
\begin{equation*}
\begin{aligned}
|\nu_{K_n}(X)-\nu(X)|&=|\nu_{K_n}(X)-\nu(X_1)|=|\nu_{K_n}(X_1)-\nu(X)|+\nu_{K_n}(X_0)\\
&\leq 4\varepsilon +2\varepsilon=6\varepsilon
\end{aligned}
\end{equation*}
for all $N\geq \max\{N_1,N_2\}$.  
\end{proof}


\section{The von Neumann dimensions of direct integrals}\label{svndim}

\subsection{The group von Neumann algebra and the trace}\label{ssvnalg}

Let $\Gamma$ be a countable group with the counting measure.  
Let $\{\delta_{\gamma}\}_{\gamma\in \Gamma}$ be the usual orthonormal basis of $l^2(\Gamma)$. 
We also let $\lambda$ and $\rho$ be the left and right regular representations of $\Gamma$ on $l^2(\Gamma)$ respectively.
For all $\gamma,\gamma'\in \Gamma$, we have
$\lambda(\gamma')\delta_{\gamma}=\delta_{\gamma'\gamma}$ and $\rho(\gamma')\delta_{\gamma}=\delta_{\gamma\gamma'^{-1}}$. 
Let $\mathcal{L}(\Gamma)$ be the strong operator closure of the complex linear span of $\lambda(\gamma)$'s (or equivalently,  $\rho(\gamma)$'s). 
This is the {\it group von Neumann algebra of $\Gamma$}. 
There is a canonical faithful normal tracial state $\tau_{\Gamma}$, or simply $\tau$, on $\mathcal{L}(\Gamma)$, which is given by
\begin{center}
$\tau(x)=\langle x\delta_e,\delta_e\rangle_{l^2(\Gamma)}$, $x\in \mathcal{L}(\Gamma)$. 
\end{center}
Hence $\mathcal{L}(\Gamma)$ is a finite von Neumann algebra (which must be of type $\text{I}$ or $\text{II}_1$).  

More generally, for a tracial von Neumann algebra $M$ with the trace $\tau$, we consider the GNS representation of $M$ on the Hilbert space constructed from the completion of $M$ with respect to the inner product $\langle x,y\rangle_{\tau}=\tau(xy^*)$. 
The underlying space will be denoted by $L^2(M,\tau)$, or simply $L^2(M)$.
 
Consider a normal unital representation $\pi\colon M\to B(H)$ with both $M$ and $H$ separable.  
There exists an isometry $u\colon H\to L^2(M)\otimes l^2(\mathbb{N})$, which commutes with the actions of $M$:
\begin{center}
$u\circ\pi(x)=(\lambda(x)\otimes\id_{l^2(\mathbb{N})} )\circ u$, $\forall x\in M$,
\end{center}
where $\lambda\colon M\mapsto L^2(M)$ denotes the left action. 
Then $p=uu^*$ is a projection in $B(L^2(M)\otimes l^2(\mathbb{N}))$ such that $H\cong p( L^2(M)\otimes l^2(\mathbb{N}))$. 
We have the following result (see \cite{APintrII1} Proposition 8.2.3). 

\begin{proposition}\label{ptrdim}
The correspondence $H\mapsto p$ above defines a bijection between the set of equivalence classes of left $M$-modules and the set of equivalence classes of projections in $(M'\cap B(L^2(M)))\otimes B(l^2(\mathbb{N}))$. 
\end{proposition}

The {\it von Neumann dimension} of the $M$-module $H$ are defined to be $(\tau\otimes \Tr)(p)$ and denoted by $\dim_M(H)$, which takes its value in $[0,\infty]$. 
We have: 
\begin{enumerate}
    \item $\dim_M(\oplus_i H_i)=\sum_i \dim_M(H_i)$. 
    \item $\dim_M(L^2(M))=1$.
\end{enumerate} 
Note $\dim_M(H)$ depends on the trace $\tau$. 
If $M$ is a finite factor, i.e., $Z(M)\cong\mathbb{C}$, there is a unique normal tracial state (see \cite{J83,MvN36}) and we further have: 
\begin{enumerate}
\setcounter{enumi}{2}
    \item $\dim_M(H)=\dim_M(H')$ if and only if $H$ and $H'$ are isomorphic as $M$-modules (provided $M$ is a factor).  
\end{enumerate}
When $M$ is not a factor, there is a $Z(M)$-valued trace which determines the isomorphism class of an $M$-module (see \cite{Bek04}). 

In the following sections, we will consider the group von Neumann algebra $\mathcal{L}(\Gamma)$ with the canonical trace $tr(x)=\langle x\delta_e,\delta_e \rangle$. 
Hence the von Neumann dimension of $\mathcal{L}(\Gamma)$ is the one uniquely determined by this trace. 
Note a discrete group $\Gamma$ is called an infinite conjugacy class (ICC) group if every nontrivial conjugacy class $C_{\gamma}=\{g\gamma g^{-1}|g\in \Gamma\}$, $\gamma\neq e$, is infinite. 
It is well-known that $\mathcal{L}(\Gamma)$ is a $\rm{II}_1$ factor if and only if $\Gamma$ is a nontrivial ICC group.

Now we consider the case that $\Gamma$ is a discrete subgroup of a locally compact unimodular type I group $G$.  
Let $\mu$ be a Haar measure of $G$. A measurable set $D\subset G$ is called a {\it fundamental domain} for $\Gamma$ if $D$ satisfies $\mu(G\backslash \cup_{\gamma\in\Gamma}\gamma D)=0$ and 
$\mu(\gamma_1 D\cap \gamma_2 D)=0$ if $\gamma_1\neq \gamma_2$ in $\Gamma$. 
In this section, we always assume $\Gamma$ is a lattice, i.e., $\mu(D)<\infty$.
The measure $\mu(D)$ is called {\it covolume} of $\Gamma$ and will be denoted by $\covol(\Gamma)$. 
Note the covolume depends on the Haar measure $\mu$ (see Remark \ref{rcovol}). 

There is a natural isomorphism $L^2(G)\cong l^2(\Gamma)\otimes L^2(D,\mu)$ given by
\begin{center}
$\phi\mapsto \sum_{\gamma\in\Gamma}\delta_{\gamma}\otimes \phi_{\gamma}$ with $\phi_{\gamma}(z)=\phi(\gamma\cdot z)$,
\end{center}
where $z\in D$ and $\gamma\in \Gamma$. 
The restriction representation $\lambda_G|_{\Gamma}$ of $\Gamma$ is the tensor product of $\lambda_{\Gamma}$ on $l^2(\Gamma)$ and the identity operator $\id$ on $L^2(D,\mu)$. 
Hence we obtain the von Neumann algebra $\lambda_G(\Gamma)''\cong \mathcal{L}(\Gamma)\otimes \mathbb{C}=\mathcal{L}(\Gamma)$, which will be denoted by $M$ throughout this section. 
Please note $L^2(M)=l^2(\Gamma)$. 



\subsection{A theorem of von Neumann dimension}\label{ssvdim}
Suppose $X$ is a measurable subset of 
$\widehat{G}$ with the Plancherel measure $\nu(X)<\infty$. 
Define
\begin{center}
$H_X=\int_{X}^{\oplus}H_{\pi}d\nu(\pi)$,   
\end{center} 
which is the direct integral of the spaces $H_{\pi}$ with $\pi\in X$. 
It is a module over $G$, its lattice $\Gamma$, and also the group von Neumann algebra $\mathcal{L}(\Gamma)$. 

We state a result on the von Neumann dimension of direct integrals. 
One may refer to \cite{yj22} Section 4 for the proof. 

\begin{theorem}\label{tdimmeas}
Let $G$ be a locally compact unimodular type I group with Haar measure $\mu$. 
Let $\nu$ be the Plancherel measure on the unitary dual $\widehat{G}$ of $G$. 
Suppose $\Gamma$ is a lattice in $G$ and $\mathcal{L}(\Gamma)$ is the group von Neumann algebra of $\Gamma$.  
Let $X\subset\widehat{G}$ such that $\nu(X)<\infty$ and $H_X=\int_X^{\oplus} H_{\pi}d\nu(\pi)$.  
We have 
\begin{center}
$\dim_{\mathcal{L}(\Gamma)}(H_X)=\covol(\Gamma)\cdot \nu(X)$. 
\end{center}
\end{theorem}

\begin{remark}\label{rcovol}
\begin{enumerate}
\item If $\mu'=k\cdot \mu$ is another Haar measure on $G$ for some $k>0$, the covolumes are related by  $\covol'(\Gamma)=\mu'(G/\Gamma)=k'\cdot\mu(G/\Gamma)=k\cdot\covol(\Gamma)$.
But the induced Plancherel measure $\nu'=k^{-1} \cdot\nu$ and the dependencies cancel out in the formula above. 
    

\item There is a relevant approach by H. Peterson and A. Valette \cite{PetVlt14}.
They study the von Neumann dimension over locally compact groups. 
The group von Neumann algebra is equipped with
a semifinite tracial weight instead of a tracial state for a discrete group. 
It is motivated by the study of $L^2$-Betti number of locally compact groups \cite{PtsH13}.

\end{enumerate}
\end{remark}

If $\pi$ is an atom in $\widehat{G}$, i.e., $\nu(\{\pi\})>0$,  
the irreducible representation $\pi$ is a discrete series and $\nu(\{\pi\})$ is just the formal dimension of $\pi$ \cite{DiCalg,Ro}. 
Under this assumption, if $G$ is a real Lie group that has discrete series and $\Gamma$ is an ICC group,  the theorem reduces to the special case of a single representation (see \cite{GHJ} Theorem 3.3.2)
\begin{center}
$\dim_{\mathcal{L}(\Gamma)}(H_{\pi})=\covol(\Gamma)\cdot d_{\pi}$.  
\end{center}
This is motivated by the geometric construction of discrete series of Lie groups by M. Atiyah and W. Schmid \cite{ASds77}.

\subsection{The proof of the main theorem}

We will prove the main theorem. 
We first give the proof for a tower of uniform lattices. 
\begin{theorem}\label{tmaincpt}[a tower of uniform lattices]
Let $\Gamma_1\supsetneq \Gamma_2 \supsetneq \cdots$ be a normal tower of cocompact lattice in a semisimple real Lie group $G$ such that $\cap_{n\geq 1} \Gamma_n=\{1\}$. 
For any bounded subset $X$ of $\widehat{G}$, we have
\begin{center}
$\lim\limits_{n\to \infty}\cfrac{m(X,\Gamma_n)}{\dim_{\mathcal{L}\Gamma_n}H_X}=1$. 
\end{center}
\end{theorem}
\begin{proof}
Recall that $m_{\Gamma_K}(X)=\vol(\Gamma_K\backslash G(F_{\infty}))\nu_K(X)$ by definition and $\dim_{\mathcal{L}\Gamma_n}H_X=\vol(\Gamma_K\backslash G(F_{\infty}))\nu(X)$ by Theorem \ref{tdimmeas}. 
We need to show $\lim\limits_{n\to \infty}\nu_{K_{n}}(X)=\nu(X)$, which reduces to
$\lim\limits_{n\to \infty}\nu_{K_{n}}(\widehat{\phi})=\phi(1)$ for all $\phi\in C^{\infty}_{c}(G(F_{\infty})^1)$ by  Theorem \ref{tsvg}. 

From Proposition \ref{ptrmk}, 
we know
\begin{equation*}
\begin{aligned}
    \tr R_{\text{disc}}(\phi\otimes \frac{1_K}{\vol(K)})&=m_K(\widehat{\phi})=\vol(G(\mathbb{Q})\backslash G(\mathbb{A})^1)/K)\cdot \nu_K(\widehat{\phi}), 
\end{aligned}
\end{equation*}
which is to say $\tr R_{\text{disc}}(\phi\otimes 1_K)=\vol(G(\mathbb{Q})\backslash G(\mathbb{A})^1))\cdot \nu_K(\widehat{\phi})$. 
By Proposition \ref{pcpttr=}, we have $\lim\limits_{n\to \infty} \tr R_{\text{disc}}(\phi\otimes 1_{K_{n}})=\vol(G(\mathbb{Q})\backslash G(\mathbb{A})^1))\cdot \phi(1)$. 
Hence $\lim\limits_{n\to \infty}\nu_{K_{n}}(\widehat{\phi})=\phi(1)$.  
\end{proof}

For the non-uniform case, the distribution $J(f)$ in Equation \ref{etf} will no longer be the trace of $R_{\text{disc}}(f)$, which leads to the main task for most arithmetic subgroups. 
Fortunately, Finis-Lapid-M\"{u}ller proved the following result on the limit of the spectral side of Equation \ref{etf} (see \cite{FLM15} Corollary 7.8). 
\begin{theorem}\label{tspeclimitFLM15}[Finis-Lapid-M\"{u}ller]
Suppose $G=\SL(n)$.
Let $\{I_n\}$ be a family of descending integral ideals in $\mathcal{O}_F$ prime to $S$ and  $K_{n}=K(I_n)$ be the compact subgroups of $G(\mathbb{A}^S)$ given by $I_n$.
We have
\begin{center}
    $\lim_{n\to \infty} J(h_S\otimes 1_{K_{n}})=\lim_{n\to \infty}\tr R_{\text{disc}}(h_S\otimes 1_{K_{n}})$ 
\end{center}
for any $h_S\in C^{\infty}_{c}(G(F_S)^1)$. 
\end{theorem}

Then we are able to prove: 

\begin{corollary}\label{tcpt}[principal congruence subgroups in $\SL(n,\mathbb{R})$]
Let $\Gamma_1\supsetneq \Gamma_2 \supsetneq \cdots$ be a tower of principal congruence subgroups in $G=\SL(n,\mathbb{R})$. 
For any bounded subset $X$ of $\widehat{G}$, we have
\begin{center}
$\lim\limits_{n\to \infty}\cfrac{m(X,\Gamma_n)}{\dim_{\mathcal{L}\Gamma_n}H_X}=1$. 
\end{center}
\end{corollary}
\begin{proof}
As shown in Theorem \ref{tmaincpt}, it suffices to prove $\lim\limits_{n\to \infty}\nu_{K_{n}}(\widehat{\phi})=\phi(1)$ for all $\phi\in C^{\infty}_{c}(G(F_{\infty})^1)$. 
By Proposition \ref{ptrmk} and Theorem \ref{tspeclimitFLM15}, we know 
\begin{center}
    $\lim\limits_{n\to \infty}\vol(G(\mathbb{Q})\backslash G(\mathbb{A})^1))\cdot \nu_{K_{n}}(\widehat{\phi})=\lim\limits_{n\to \infty} \tr R_{\text{disc}}(\phi\otimes 1_{K_n})=\lim\limits_{n\to \infty} J(\phi\otimes 1_{K_n})$. 
\end{center}
As $\lim\limits_{n\to \infty}\frac{\vol(G(F)\backslash G(\mathbb{A})^1)\phi(1)}{J(\phi\otimes 1_{K_{n}})}=1$ in Corollary \ref{cgeolimit}, 
we obtain $\lim\limits_{n\to \infty}\nu_{K_n}(\widehat{\phi})=\phi(1)$.  
\end{proof}

\bibliographystyle{abbrv}
\typeout{}
\bibliography{MyLibrary} 

\begin{thebibliography}{10}

\bibitem{ABBGNRS17}
M.~Abert, N.~Bergeron, I.~Biringer, T.~Gelander, N.~Nikolov, J.~Raimbault, and
  I.~Samet.
\newblock On the growth of {$L^2$}-invariants for sequences of lattices in
  {L}ie groups.
\newblock {\em Ann. of Math. (2)}, 185(3):711--790, 2017.

\bibitem{APintrII1}
C.~Anantharaman and S.~Popa.
\newblock An introduction to $\text{II}_1$ factors.
\newblock {\em preprint}, 8, 2017.

\bibitem{ArtI}
J.~Arthur.
\newblock A trace formula for reductive groups. {I}. {T}erms associated to
  classes in {$G({\bf Q})$}.
\newblock {\em Duke Math. J.}, 45(4):911--952, 1978.

\bibitem{ArtII}
J.~Arthur.
\newblock A trace formula for reductive groups. {II}. {A}pplications of a
  truncation operator.
\newblock {\em Compositio Math.}, 40(1):87--121, 1980.

\bibitem{ArtUni}
J.~Arthur.
\newblock A measure on the unipotent variety.
\newblock {\em Canad. J. Math.}, 37(6):1237--1274, 1985.

\bibitem{ArtIntro}
J.~Arthur.
\newblock An introduction to the trace formula.
\newblock In {\em Harmonic analysis, the trace formula, and {S}himura
  varieties}, volume~4 of {\em Clay Math. Proc.}, pages 1--263. Amer. Math.
  Soc., Providence, RI, 2005.

\bibitem{ASds77}
M.~Atiyah and W.~Schmid.
\newblock A geometric construction of the discrete series for semisimple {L}ie
  groups.
\newblock {\em Invent. Math.}, 42:1--62, 1977.

\bibitem{Bek04}
B.~Bekka.
\newblock Square integrable representations, von {N}eumann algebras and an
  application to {G}abor analysis.
\newblock {\em J. Fourier Anal. Appl.}, 10(4):325--349, 2004.

\bibitem{BoGa83}
A.~Borel and H.~Garland.
\newblock Laplacian and the discrete spectrum of an arithmetic group.
\newblock {\em Amer. J. Math.}, 105(2):309--335, 1983.

\bibitem{Corw77}
L.~Corwin.
\newblock The {P}lancherel measure in nilpotent {L}ie groups as a limit of
  point measures.
\newblock {\em Math. Z.}, 155(2):151--162, 1977.

\bibitem{dGW78}
D.~L. de~George and N.~R. Wallach.
\newblock Limit formulas for multiplicities in {$L^{2}(\Gamma \backslash G)$}.
\newblock {\em Ann. of Math. (2)}, 107(1):133--150, 1978.

\bibitem{dGW79}
D.~L. DeGeorge and N.~R. Wallach.
\newblock Limit formulas for multiplicities in {$L^{2}(\Gamma \backslash G)$}.
  {II}. {T}he tempered spectrum.
\newblock {\em Ann. of Math. (2)}, 109(3):477--495, 1979.

\bibitem{Delm86}
P.~Delorme.
\newblock Formules limites et formules asymptotiques pour les multiplicit\'{e}s
  dans {$L^2(G/\Gamma)$}.
\newblock {\em Duke Math. J.}, 53(3):691--731, 1986.

\bibitem{DiaShur05}
F.~Diamond and J.~Shurman.
\newblock {\em A first course in modular forms}, volume 228 of {\em Graduate
  Texts in Mathematics}.
\newblock Springer-Verlag, New York, 2005.

\bibitem{DiCalg}
J.~Dixmier.
\newblock {\em {$C\sp*$}-algebras}.
\newblock North-Holland Mathematical Library, Vol. 15. North-Holland Publishing
  Co., Amsterdam-New York-Oxford, 1977.
\newblock Translated from the French by Francis Jellett.

\bibitem{FLM11b}
T.~Finis and E.~Lapid.
\newblock On the spectral side of {A}rthur's trace formula---combinatorial
  setup.
\newblock {\em Ann. of Math. (2)}, 174(1):197--223, 2011.

\bibitem{FLM18}
T.~Finis and E.~Lapid.
\newblock An approximation principle for congruence subgroups {II}: application
  to the limit multiplicity problem.
\newblock {\em Math. Z.}, 289(3-4):1357--1380, 2018.

\bibitem{FLM11a}
T.~Finis, E.~Lapid, and W.~M\"{u}ller.
\newblock On the spectral side of {A}rthur's trace formula---absolute
  convergence.
\newblock {\em Ann. of Math. (2)}, 174(1):173--195, 2011.

\bibitem{FLM11}
T.~Finis, E.~Lapid, and W.~M\"{u}ller.
\newblock On the spectral side of {A}rthur's trace formula---absolute
  convergence.
\newblock {\em Ann. of Math. (2)}, 174(1):173--195, 2011.

\bibitem{FLM15}
T.~Finis, E.~Lapid, and W.~M\"{u}ller.
\newblock Limit multiplicities for principal congruence subgroups of {${\rm
  GL}(n)$} and {${\rm SL}(n)$}.
\newblock {\em J. Inst. Math. Jussieu}, 14(3):589--638, 2015.

\bibitem{GelTF}
S.~Gelbart.
\newblock {\em Lectures on the {A}rthur-{S}elberg trace formula}, volume~9 of
  {\em University Lecture Series}.
\newblock American Mathematical Society, Providence, RI, 1996.

\bibitem{Gel75}
S.~S. Gelbart.
\newblock {\em Automorphic forms on ad\`ele groups}.
\newblock Annals of Mathematics Studies, No. 83. Princeton University Press,
  Princeton, N.J.; University of Tokyo Press, Tokyo, 1975.

\bibitem{GHJ}
F.~M. Goodman, P.~de~la Harpe, and V.~F.~R. Jones.
\newblock {\em Coxeter graphs and towers of algebras}, volume~14 of {\em
  Mathematical Sciences Research Institute Publications}.
\newblock Springer-Verlag, New York, 1989.

\bibitem{Ji98}
L.~Ji.
\newblock The trace class conjecture for arithmetic groups.
\newblock {\em J. Differential Geom.}, 48(1):165--203, 1998.

\bibitem{J83}
V.~F.~R. Jones.
\newblock Index for subfactors.
\newblock {\em Invent. Math.}, 72(1):1--25, 1983.

\bibitem{Kn}
A.~W. Knapp.
\newblock {\em Representation theory of semisimple groups}, volume~36 of {\em
  Princeton Mathematical Series}.
\newblock Princeton University Press, Princeton, NJ, 1986.
\newblock An overview based on examples.

\bibitem{Kn97llds}
A.~W. Knapp.
\newblock Introduction to the {L}anglands program.
\newblock In {\em Representation theory and automorphic forms ({E}dinburgh,
  1996)}, volume~61 of {\em Proc. Sympos. Pure Math.}, pages 245--302. Amer.
  Math. Soc., Providence, RI, 1997.

\bibitem{Kn97tf}
A.~W. Knapp.
\newblock Theoretical aspects of the trace formula for {${\rm GL}(2)$}.
\newblock In {\em Representation theory and automorphic forms ({E}dinburgh,
  1996)}, volume~61 of {\em Proc. Sympos. Pure Math.}, pages 355--405. Amer.
  Math. Soc., Providence, RI, 1997.

\bibitem{MvN36}
F.~J. Murray and J.~Von~Neumann.
\newblock On rings of operators.
\newblock {\em Ann. of Math. (2)}, 37(1):116--229, 1936.

\bibitem{PtsH13}
H.~D. Petersen.
\newblock L2-betti numbers of locally compact groups.
\newblock {\em Comptes Rendus Mathematique}, 351(9-10):339--342, 2013.

\bibitem{PetVlt14}
H.~D. Petersen and A.~Valette.
\newblock {$L^2$}-{B}etti numbers and {P}lancherel measure.
\newblock {\em J. Funct. Anal.}, 266(5):3156--3169, 2014.

\bibitem{PR94}
V.~Platonov and A.~Rapinchuk.
\newblock {\em Algebraic groups and number theory}, volume 139 of {\em Pure and
  Applied Mathematics}.
\newblock Academic Press, Inc., Boston, MA, 1994.
\newblock Translated from the 1991 Russian original by Rachel Rowen.

\bibitem{Ro}
A.~Robert.
\newblock {\em Introduction to the representation theory of compact and locally
  compact groups}, volume~80 of {\em London Mathematical Society Lecture Note
  Series}.
\newblock Cambridge University Press, Cambridge-New York, 1983.

\bibitem{RoSp87}
J.~Rohlfs and B.~Speh.
\newblock On limit multiplicities of representations with cohomology in the
  cuspidal spectrum.
\newblock {\em Duke Math. J.}, 55(1):199--211, 1987.

\bibitem{Svg97}
F.~Sauvageot.
\newblock Principe de densit\'{e} pour les groupes r\'{e}ductifs.
\newblock {\em Compositio Math.}, 108(2):151--184, 1997.

\bibitem{Sav89}
G.~Savin.
\newblock Limit multiplicities of cusp forms.
\newblock {\em Invent. Math.}, 95(1):149--159, 1989.

\bibitem{Shin12}
S.~W. Shin.
\newblock Automorphic {P}lancherel density theorem.
\newblock {\em Israel J. Math.}, 192(1):83--120, 2012.

\bibitem{yj22}
J.~Yang.
\newblock Plancherel measures of reductive adelic groups and von neumann
  dimensions.
\newblock {\em arXiv preprint arXiv:2203.07974}, 2022.

\end{thebibliography}

\textit{E-mail address}: \href{mailto:junyang@fas.harvard.edu}{junyang@fas.harvard.edu}\\

{Harvard University, Cambridge, MA 02138, USA}

\end{document}